\documentclass[11pt]{article}
\usepackage{amsmath,amsthm,amssymb,amsfonts}
\usepackage{latexsym}
\usepackage{graphicx,psfrag,import}
\usepackage{fullpage}
\usepackage{framed}
\usepackage{verbatim}
\usepackage{color}
\usepackage{epsfig}
\usepackage{epstopdf}
\usepackage{hyperref}
\usepackage{geometry}
\usepackage{mathtools}
\usepackage{nonfloat}
\usepackage{enumerate}
\usepackage{multicol}
\usepackage{a4wide}
\usepackage{booktabs}
\usepackage{enumitem}
\usepackage{lineno}
\usepackage{parcolumns}
\usepackage{thmtools}
\usepackage{xr}
\usepackage{epstopdf}
\usepackage{mathrsfs}
\usepackage{subfig}
\usepackage{caption}
\usepackage{comment}
\usepackage{authblk}
\usepackage{setspace}

\newcommand\myfigure[1]{
\medskip\noindent\begin{minipage}{\columnwidth}
\centering
\end{minipage}\medskip}

\theoremstyle{plain}
\newtheorem{theorem}{Theorem}[section]

\newtheorem{lemma}[theorem]{Lemma}

\theoremstyle{definition}
\newtheorem{definition}[theorem]{Definition}
\newtheorem{example}[theorem]{Example}

\newtheorem{remark}[theorem]{Remark}

\theoremstyle{remark}

\newcommand\RR{\mathbb{R}}

\newcommand\x{\boldsymbol{x}}

\newcommand\w{\boldsymbol{w}}
\newcommand\s{\boldsymbol{s}}
\newcommand\z{\boldsymbol{z}}
\newcommand\X{\boldsymbol{X}}
\newcommand\bv{\boldsymbol{v}}
\newcommand\bM{\boldsymbol{M}}
\newcommand\bn{\boldsymbol{n}}
\newcommand\bu{\boldsymbol{u}}

\def\GG{\mathcal{G}}

        {\begin{list}
                {\noindent\makeboX[0mm][r]{$\bullet$}}
                {\leftmargin=5.5eX \usecounter{enumi}

      \topsep=1.5mm \itemsep=-.75eX}
        }
        {\end{list}}

\begin{document}

\title{Endotactic Networks and Toric Differential Inclusions}
\author[1]{Gheorghe Craciun}
\author[2]{Abhishek Deshpande}
\affil[1]{Department of Mathematics and Department of Biomolecular Chemistry, University of Wisconsin-Madison, {\tt craciun@math.wisc.edu}.}
\affil[2]{Department of Mathematics, University of Wisconsin-Madison, {\tt deshpande8@wisc.edu}.}
\maketitle

\begin{abstract}
An important dynamical property of biological interaction networks is {\em persistence}, which intuitively means that ``no species goes extinct". It has been conjectured that  dynamical system models of weakly reversible networks (i.e., networks for which each reaction is part of a cycle) are persistent. The property of persistence is also related to the well known {\em global attractor conjecture}. An approach for the proof of the global attractor conjecture uses an embedding of weakly reversible dynamical systems into {\em toric differential inclusions}. We show that the larger class of \textit{endotactic dynamical systems} can also be embedded into toric differential inclusions. Moreover, we show that, essentially, endotactic networks form the largest class of networks with this property.

\end{abstract}

\section{Introduction}

Mathematical models of biological interaction networks are often power-law dynamical systems, given by
\begin{eqnarray}\label{eq:basic_polynomial}
\frac{d\x}{dt} = \displaystyle\sum_{i=1}^r \x^{\s_i}\w_i,
\end{eqnarray}

\noindent
where $\x \in \mathbb{R}^n_{>0}$, $\s_i, \w_i \in \mathbb{R}^n$, and $\x^{\s} = x_1^{s_1}  x_2^{s_2} ... x_n^{s_n}$. In particular, all polynomial dynamical systems are of the form given by (\ref{eq:basic_polynomial}). A key question in this context is the following: Given an initial condition $\x(0)\in\mathbb{R}^n_{>0}$, is it true that none of the species become extinct, i.e., $\displaystyle\liminf\limits_{t\rightarrow\infty}\x_i(t) > 0$ for all $i$? It has been conjectured that if the underlying network is weakly reversible, then no species can become extinct~\cite{feinberg1987chemical}. This has been called the \textit{Persistence Conjecture}~\cite{craciun2013persistence} and has been an open problem for several decades. The property of persistence is related to another conjecture called the \textit{Global Attractor Conjecture}, which asserts the existence of a unique globally attracting equilibrium for complex balanced systems~\cite{craciun2009toric}. Several special cases of this conjecture have been proved using a variety of approaches~\cite{anderson2011proof,craciun2013persistence,pantea2012persistence,gopalkrishnan2014geometric}, and a proof of this conjecture in full generality has been proposed recently in~\cite{craciun2015toric}. A key component of this proof relies on the machinery of differential inclusions to build \textit{zero-separating surfaces} and eventually prove the existence of a unique globally attracting steady state. More specifically, the proof uses the embedding of weakly reversible dynamical systems into \textit{toric differential inclusions}. The main contribution of this paper is to show that a much larger class of dynamical systems called \textit{endotactic dynamical systems} can also be embedded into toric differential inclusions. In addition, we show that endotactic dynamical systems are, essentially, the largest class of dynamical systems that possess this property.

This paper is structured as follows: In Section~\ref{sec:E-graph_endotactic}, we introduce \textit{reaction networks}, which are directed graphs in Euclidean space, called \textit{Euclidean embedded graphs} (or \textit{E-graphs}). In particular, we define 
\textit{endotactic E-graphs}, which can be thought of as ``inward pointing" reaction networks. We also introduce dynamical systems generated by E-graphs. To incorporate the uncertainty associated with measuring the rate constants in the network, we allow them to take values in a bounded interval, and we refer to these more general (non-autonomous) systems  as \textit{variable-$k$ dynamical systems}. In Section~\ref{sec:persistence_permanence_GAC}, we introduce the notions of persistence and permanence and describe their connection to the Global Attractor Conjecture. In Section~\ref{sec:polyhedral_fan_TDI}, we describe polyhedral fans and toric differential inclusions. In Section~\ref{sec:embed_endotactic_TDI}, we state our main result, given by Theorem~\ref{thm:embed_endotactic_tdi}, which says that any variable-$k$ endotactic dynamical system can be embedded into a toric differential inclusion. In Section~\ref{sec:converse_endotactic_TDI}, in Theorem~\ref{thm:converse_endotactic}, we give a converse to Theorem~\ref{thm:embed_endotactic_tdi}. More precisely, we show that if any variable-$k$ dynamical system generated by an E-graph can be embedded into a toric differential inclusion, then this E-graph must be endotactic. We conclude that among all E-graphs, it is \textit{endotactic E-graphs} that are most naturally associated with toric differential inclusions.

\section{Euclidean embedded graphs and endotactic networks}\label{sec:E-graph_endotactic}

An Euclidean embedded graph (abbreviated as E-graph) is a finite directed graph $\GG=(V,E)$, where $V\subset\mathbb{R}^n$ is the set of vertices and $E$ is the set of edges~\cite{craciun2015toric,craciun2019polynomial}. Every edge $\s_i\rightarrow\s'_i\in E$ has a \textit{source} vertex $\s_i\in V$, a \textit{target} vertex $\s'_i\in V$ and an \textit{edge vector} $\s'_i-\s_i$. We will also write  $\s_i \to \s_i' \in E$. Figure~\ref{fig:euclidean_embedded_graph} shows a few examples of Euclidean embedded graphs. A reaction network can be regarded as a Euclidean embedded graph $\GG=(V,E)$, where $E$ is the set of reactions~\cite{craciun2019polynomial}.  

\begin{figure}[ht!]
\centering
\includegraphics[scale=0.4]{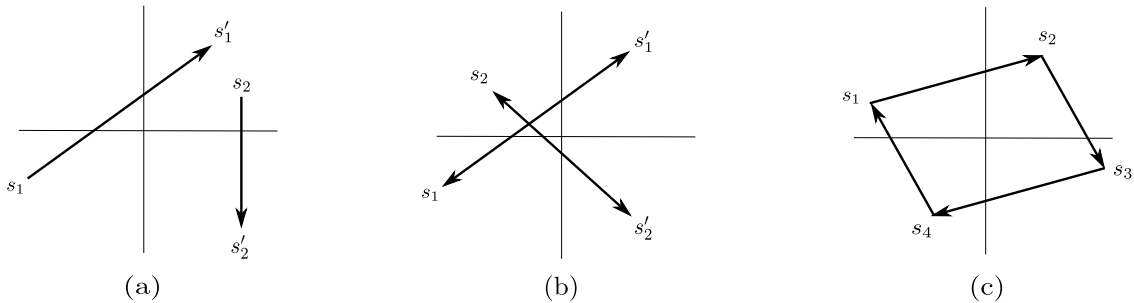}
\caption{\small Examples of E-graphs.}
\label{fig:euclidean_embedded_graph}
\end{figure} 

\noindent
An E-graph $\GG=(V,E)$ is \textit{reversible} if for every edge $\s_i\rightarrow \s'_i\in E$, we also have $\s'_i\rightarrow \s_i\in E$. An E-graph is \textit{weakly reversible} if every edge in $E$ is part of some cycle. Figure~\ref{fig:euclidean_embedded_graph}(b) shows a reversible E-graph, and  Figure~\ref{fig:euclidean_embedded_graph}(c) shows a weakly reversible E-graph. Note that every reversible E-graph is weakly reversible. In this paper, we are interested in {\em endotactic} E-graphs, which include the set of weakly reversible E-graphs. Endotactic E-graphs have originally been introduced in~\cite{craciun2013persistence}, where persistence of two-dimensional endotactic networks was proved. Since then, several equivalent formulations of endotactic E-graphs have been given, including some that have been used for checking the condition of endotacticity algorithmically~\cite{johnston2016computational}. We use the formulation presented in~\cite{johnston2016computational} as our definition of an endotactic E-graph:

\begin{definition}\label{def:endotactic}
Let $\GG=(V,E)$ be an E-graph. Then $\GG$ is \textit{endotactic} if for every $\w\in\RR^n$ and $\s_i\rightarrow\s'_i\in E$ such that $\w\cdot(\s'_i - \s_i)<0$, there exists $\s_j\rightarrow\s'_j\in E$ such that $\w\cdot (\s'_j - \s_j)>0$ and $\w\cdot \s_j < \w\cdot \s_i$.
\end{definition} 
Note that an E-graph is endotactic iff all its one-dimensional projections are endotactic~\cite{craciun2013persistence}.
For two dimensional E-graphs, it suffices to check if its one-dimensional projections are endotactic for a finite set of vectors given by (i) the inward pointing normals to the boundary of the convex hull of source vertices and (ii) the vectors corresponding to the cartesian axes~\cite{craciun2013persistence}. Other equivalent methods of checking whether a network is endotactic include the ``parallel sweep test" introduced in~\cite{craciun2013persistence}. See Figure~\ref{fig:network_types} for examples.

\begin{figure}[ht!]
\centering
\includegraphics[scale=0.5]{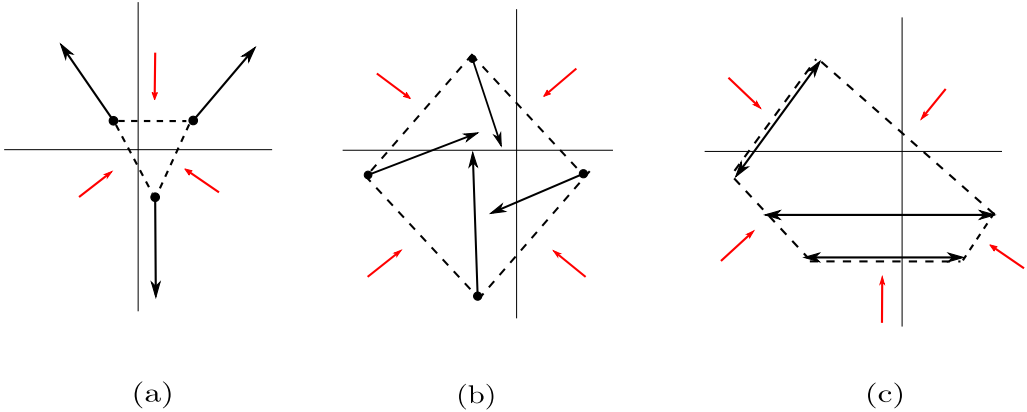}
\caption{\small Other examples of E-graphs. Dotted black lines indicate the convex hull of the set of source vertices. Red arrows are inward pointing normals to this convex hull. (a) This E-graph is not endotactic, since its one-dimensional projection on at least one of the inward pointing red normals (in fact all of them in this case) is not endotactic. (b) This E-graph is endotactic since its one-dimensional projections on the red vectors and cartesian axes are endotactic~\cite{craciun2013persistence}. (c) This E-graph is endotactic since any reversible (or weakly reversible) E-graph is endotactic~\cite{craciun2013persistence}.}
\label{fig:network_types}
\end{figure} 

Associated with each E-graph is a family of dynamical systems. Inspired by {\em mass action kinetics}~\cite{guldberg1864studies,craciun2013persistence,feinberg1987chemical,yu2018mathematical,adleman2014mathematics}, a dynamical system generated by $\GG$ is given by\footnote{Note that (\ref{eq:mass_action_dynamics_autonomous}) is a mass-action system if $\s_i,\s'_i\in\mathbb{R}^n_{\geq 0}$ for all $\s_i\rightarrow \s'_i\in E$.}
\begin{eqnarray}\label{eq:mass_action_dynamics_autonomous}
\frac{d\x}{dt}=\displaystyle\sum_{\s_i \to \s'_i\in E} k_{\s_i\rightarrow \s'_i}\x^{\s_i}(\s'_i-\s_i),
\end{eqnarray}
where $k_{\s_i\rightarrow \s'_i}$ is a (positive) rate constant corresponding to the edge $\s_i\rightarrow \s'_i$. A \textit{weakly reversible dynamical system} is a dynamical system given by (\ref{eq:mass_action_dynamics_autonomous}) for some weakly reversible E-graph $\GG=(V,E)$. An equilibrium point of (\ref{eq:mass_action_dynamics_autonomous}) is a point $\x_0\in\mathbb{R}^n_{\geq 0}$ such that the right-hand side of (\ref{eq:mass_action_dynamics_autonomous}) evaluated at $\x_0$ is zero. A \textit{complex balanced equilibrium} of (\ref{eq:mass_action_dynamics_autonomous}) is a point $\x_0\in\mathbb{R}^n_{\geq 0}$ such that for any vertex $\s_0$ of $\GG$, we have
\begin{eqnarray}
\displaystyle\sum_{\s_0\rightarrow \s'_0 \in E}k_{\s_0\rightarrow \s'_0}\x^{\s_0} = \displaystyle\sum_{\s'_0\rightarrow \s_0\in E}k_{\s'_0\rightarrow \s_0}\x^{\s'_0}.
\end{eqnarray}
A \textit{complex balanced dynamical system} is a dynamical system given by (\ref{eq:mass_action_dynamics_autonomous}) that admits a positive complex balanced equilibrium.\footnote{If we assume that all the vertices in $\GG$ have non-negative integer coordinates, then complex balanced dynamical system are known as ``toric dynamical systems"~\cite{craciun2009toric}.}

More generally, rate constants in (\ref{eq:mass_action_dynamics_autonomous}) could be time-dependent~\cite{craciun2013persistence}, to account for external factors that may influence the dynamics of the network. Time-dependent rate constants have also been considered in~\cite{gopalkrishnan2014geometric} and have been called a ``tempering". In this framework, the E-graph $\GG$ generates non-autonomous dynamical systems of the form

\begin{eqnarray}\label{eq:mass_action_dynamics}
\frac{d\x}{dt}=\displaystyle\sum_{\s_i \to \s'_i\in E} k_{\s_i\rightarrow \s'_i}(t)\x^{\s_i}(\s'_i-\s_i).
\end{eqnarray}

\noindent A dynamical system of the form (\ref{eq:mass_action_dynamics}) is called a {\em variable-$k$ power law dynamical system}\footnote{Some previous works~\cite{craciun2013persistence,johnston2016computational,pantea2012persistence} used ``$k$-variable" instead of ``variable-$k$".} if there exists an $\epsilon>0$ such that $\epsilon\leq k_{\s_i\rightarrow \s'_i}(t) \leq\frac{1}{\epsilon}$ for every $\s_i\rightarrow \s'_i\in E$~\cite{craciun2015toric}. Note that we have not placed any restriction on the vectors $\s_i$. The results in this paper hold in the more general setting of power law systems. An \textit{endotactic variable-k dynamical system} is a dynamical system given by (\ref{eq:mass_action_dynamics}) for some endotactic E-graph $\GG=(V,E)$.

\section{Persistence, permanence and the global attractor conjecture}\label{sec:persistence_permanence_GAC} 

Consider an E-graph $\GG=(V,E)$ and a dynamical system represented by (\ref{eq:mass_action_dynamics}). The \textit{stoichiometric subspace} of $\GG$ is 
\begin{eqnarray}
S=\text{span}\{\s'_i-\s_i\mid \s_i\rightarrow \s'_i \in E\}.
\end{eqnarray}
Note that the right hand side of (\ref{eq:mass_action_dynamics}) belongs to $S$, for all $\x$ and $t$. Given $\x^*\in\RR^n_{>0}$, the \textit{stoichiometric compatibility class} of $\x^*$ is the \textit{invariant} polyhedron 
\begin{eqnarray}
C_{\x^*}=(\x^*+S)\,\cap\,\RR^n_{>0}.
\end{eqnarray}
Let $\x(t)$ be a solution of the dynamical system (\ref{eq:mass_action_dynamics}). A dynamical system is said to be \textit{persistent} if for any initial condition $\x(0)\in\RR^n_{>0}$, we have 
\begin{eqnarray}
\displaystyle\liminf_{t\rightarrow T_{\x(0)}}\x_i(t)>0
\end{eqnarray}
for all $i=1,2,...,n$, where $T_{\x(0)}$ is the largest $T$ such that $\x(t)$ is well defined on $[0,T)$. A dynamical system is \textit{permanent} if for any initial condition $\x(0)\in\RR^n_{>0}$, there exists a compact set $M\subset C_{\x(0)}$ such that the trajectory satisfies $\x(t)\in M$ for $t$ large enough. It follows that if a dynamical system is permanent, then it is persistent. A point $\x^*\in\RR^n_{>0} $ is called a \textit{global attractor} within its stoichiometric compatibility class $C_{\x^*}$ if 
\begin{eqnarray}
\displaystyle\lim_{t\rightarrow\infty}\x(t)= \x^*
\end{eqnarray}
for all solutions $\x(t)$ with $\x(0)\in C_{\x^*}$. We now have enough terminology to state the main open problems in this field~\cite{craciun2013persistence}:\\

\noindent\textbf{Persistence Conjecture:} \textit{Weakly reversible dynamical systems are persistent.} \\

\noindent\textbf{Extended Persistence Conjecture:} \textit{Endotactic variable-$k$ dynamical systems are persistent.} \\

\noindent\textbf{Permanence Conjecture:} \textit{Weakly reversible dynamical systems are permanent.} \\

\noindent\textbf{Extended Permanence Conjecture:} \textit{Endotactic variable-$k$ dynamical systems are permanent.} \\

\noindent\textbf{Global Attractor Conjecture:} \textit{For complex-balanced dynamical systems, each stoichiometric compatibility class has a unique globally attracting equilibrium.} \\

It is known that the Extended Permanence Conjecture implies all the other conjectures. In particular, the Extended Permanence Conjecture implies the Permanence conjecture, which implies the Persistence conjecture. Further, it can be shown that persistence of variable-$k$ weakly reversible dynamical systems in dimension $n$ implies the Global Attractor Conjecture in dimension $n+1$~\cite{craciun2013persistence,craciun2015toric}. We now describe various results towards the proof of these conjectures.

The persistence properties of dynamical systems on the positive orthant are strongly related to the existence of $\omega$-limit points on the boundary of the positive orthant. It is known from~\cite{siegel2000global,sontag2001structure} that $\omega$-limit points of complex balanced dynamical systems must be equilibrium points (either positive or on the boundary). Among the equilibrium points on the boundary, the vertices~\cite{anderson2008global,craciun2009toric} and facet interiors~\cite{anderson2010dynamics} of the stoichiometric compatibility classes are forbidden from containing $\omega$-limit points. In addition, Angeli, De Leenheer and Sontag have shown that for conservative networks, no boundary equilibria can be $\omega$-limit points if every \textit{siphon} contains a linear conserved quantitiy~\cite{angeli2007petri}. For more details about how siphons are related to the structure of reaction networks, see~\cite{gopalkrishnan2011catalysis,deshpande2014autocatalysis}.

Anderson~\cite{anderson2011proof} has proved the Global Attractor Conjecture for networks with a single connected component. Moreover, Craciun, Nazarov and Pantea~\cite{craciun2013persistence} have proved the Extended Permanence Conjecture for the two dimensional case and have used this to prove the Global Attractor Conjecture in the three dimensional case. These results have been extended by Pantea~\cite{pantea2012persistence} to endotactic dynamical systems with two dimensional stoichiometric subspace and to complex balanced dynamical systems with three dimensional stoichiometric subspace. Gopalkrishnan, Miller and Shiu~\cite{gopalkrishnan2014geometric} have proved the permanence of variable-$k$ {\em strongly} endotactic networks in any dimension. More recently, Craciun~\cite{craciun2015toric} has proposed a proof of the Global Attractor Conjecture in full generality by using an embedding of weakly reversible dynamical systems into toric differential inclusions.

As shown in~\cite{craciun2015toric}, toric differential inclusions can be used in the construction of invariant regions, and invariant regions can be used to prove persistence and permanence~\cite{craciun2013persistence,pantea2012persistence,gopalkrishnan2014geometric} of some classes of power-law dynamical systems. Therefore, the results presented here point to a future avenue for proving the Extended Persistence and Extended Permanence Conjectures. 

\section{Polyhedral fans and toric differential inclusions}\label{sec:polyhedral_fan_TDI}

Polyhedral cones form the basic building blocks of polyhedral fans and toric differential inclusions. We will define these notions explicitly in this section. A set $C$ is a polyhedral cone~\cite{rockafellar2015convex,ziegler2012lectures} if there exists a finite set of vectors $\bv_1,\bv_2,...,\bv_m$ such that 
\begin{eqnarray}
C=\left\{\displaystyle\sum_{i=1}^m\alpha_i\bv_i\mid \alpha_i\in\RR_{\geq 0}\right\}.
\end{eqnarray}
Equivalently, $C$ can be written as a finite intersection of half-spaces. A \textit{face} of a polyhedral cone is the intersection of the cone with a supporting hyperplane. A \textit{facet} is a face of codimension 1. We now have enough terminology to define a polyhedral fan.

\begin{definition}
\cite{fulton1993introduction}\ Consider a finite collection $\mathcal{F}=\{C_1,C_2,..., C_r\}$ of polyhedral cones in $\mathbb{R}^n$. Then $\mathcal{F}$ is a \textit{polyhedral fan} if (i) any face of a cone $C_i\in\mathcal{F}$ also belongs to $\mathcal{F}$ and (ii) given any two cones $C_i,C_j\in\mathcal{F}$, their intersection $C_i\cap C_j$ is a face of both $C_i$ and $C_j$.
\end{definition} 
\noindent $\mathcal{F}$ is a \textit{complete polyhedral fan} if $\displaystyle\bigcup_{C_i\in\mathcal{F}}C_i=\mathbb{R}^n$. The following lemma makes an useful observation about complete polyhedral fans.
\begin{lemma}\label{lem:maximal_cones}
Let $\mathcal{F}$ be a complete polyhedral fan in $\mathbb{R}^n$. Consider a cone $C_0\in\mathcal{F}$. Then $C_0$ is the intersection of all the maximal cones in $\mathcal{F}$ that contain it.
\end{lemma}

\begin{proof}
If $C_0\in\mathcal{F}$ is a maximal cone, then we are done. If $C_0$ is not maximal, then let $C^n_0\in\mathcal{F}$ be a maximal cone containing $C_0$. Note that $C_0$ is a face of $C^n_0$. Therefore, using the fact that any face of a cone is the intersection of the facets of the cone that contain it~\cite{fulton1993introduction} we get that $C_0=\displaystyle\bigcap_{i=1}^k C^{n-1}_i$, where each $C^{n-1}_i$ is a facet of $C^n_0$. Since $\mathcal{F}$ is a complete fan, each cone $C^{n-1}_i$ is the intersection of the two maximal cones that contain it. Then, it follows that $C_0$ is the intersection of all the maximal cones in $\mathcal{F}$ that contain it.
\end{proof}

\begin{remark}
Note that Lemma~\ref{lem:maximal_cones} implies that the set of maximal cones of a complete polyhedral fan $\mathcal{F}$ determine $\mathcal{F}$ uniquely. 
\end{remark}

\begin{definition}\label{def:hyperplane-generated-fan}
Consider a set $\mathcal{H}$ of hyperplanes in $\RR^n$ containing the origin. The \textit{hyperplane-generated polyhedral fan} $\mathcal{F}_{\mathcal{H}}$ is the complete polyhedral fan whose maximal cones are the minimal $n$-dimensional intersections of half-spaces generated by hyperplanes in $\mathcal{H}$.
\end{definition}

\noindent Figure~\ref{fig:hyperplane_fan} gives examples of hyperplane-generated polyhedral fans in $\mathbb{R}^2$ and $\mathbb{R}^3$. Note that in general, an entire hyperplane is not a cone in the polyhedral fan (unless the fan consists of exactly one hyperplane, in which case the cones in the fan are the hyperplane and the two-half spaces adjoining it), since its intersection with a maximal cone is not a face of the hyperplane. 

\begin{figure}[ht!]
\centering
\includegraphics[scale=0.37]{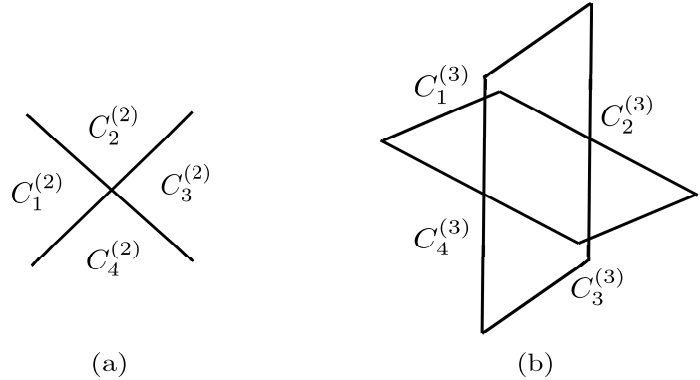}
\caption{\small Illustration of cones in hyperplane-generated polyhedral fans. (a) This polyhedral fan (in $\mathbb{R}^2$) consists of four maximal cones denoted by $C^{(2)}_i$, four one-dimensional cones (the four rays), and one cone of dimension $0$ (the origin). (b) This fan (in $\mathbb{R}^3$) has four maximal cones denoted by $C^{(3)}_i$, four two-dimensional cones and one one-dimensional cone.}
\label{fig:hyperplane_fan}
\end{figure} 

Another notion that we need to introduce is the \textit{polar} of a cone~\cite{rockafellar2015convex,ziegler2012lectures}. The polar of a cone $C$, denoted by $C^o$, is given by: 
\begin{eqnarray}\label{eq:polar_cone}
C^o=\{\z\in\mathbb{R}^n\mid \z\cdot \x\leq 0\,\text{for every}\ \x\in C\}.
\end{eqnarray}

A differential inclusion is a dynamical system of the form
\begin{eqnarray}
\frac{d\x}{dt}\in F(\x),
\end{eqnarray}
where $F$ is a set-valued map.
The stage is now set to define toric differential inclusions, which have been introduced for the purpose of proving  the global attractor conjecture~\cite{craciun2015toric}. For more examples on the use of toric differential inclusions in this context, see~\cite{brunner2018robust}.

\begin{definition}\label{def:toric_differential_inclusion}
A \emph{toric differential inclusion}~\cite{craciun2015toric} is a differential inclusion on $\RR^n_{>0}$ given by a complete polyhedral fan $\mathcal{F}$ and a real number $\delta>0$ according to the formula

\begin{eqnarray}
\frac{d\x}{dt} \in F_{\mathcal{F},\delta}(\log(\x)),
\end{eqnarray}
\noindent where 
\begin{eqnarray}\label{eq:rhs_tdi}
F_{\mathcal{F},\delta}(\X)=\displaystyle\sum_{\substack{C\in\mathcal{F}\\dist(\X,C)\leq\delta}} C^o,
\end{eqnarray}
\noindent for any $\X\in\mathbb{R}^n$.
\end{definition}
\noindent We will denote by $\mathcal{T}_{\mathcal{F},\delta}$ the toric differential inclusion generated by $\mathcal{F}$ and $\delta$.
\bigskip
\noindent From~\cite{rockafellar2015convex}, we know that 
\begin{eqnarray}
\displaystyle\sum_{\substack{C\in\mathcal{F}\\dist(\X,C)\leq\delta}} C^o=\left({\displaystyle\bigcap_{\substack{C\in\mathcal{F}\\dist(\X,C)\leq\delta}} C}\right)^o.
\end{eqnarray}
The right hand side of the toric differential inclusion is therefore given by $F_{\mathcal{F},\delta}(\X) = \mathcal{P}^o
$, where 
\begin{eqnarray}\label{eq:useful_toric}
\mathcal{P}=\displaystyle\bigcap_{\substack{C\in\mathcal{F}\\dist(\X,C)\leq\delta}} C.
\end{eqnarray}
Note that if $C_1\subseteq C_2$ and $dist(\X,C_1)\leq\delta$, then $dist(\X,C_2)\leq\delta$. Therefore, using Lemma~\ref{lem:maximal_cones}, we can rewrite the right-hand side of the toric differential inclusion as
\begin{eqnarray}
F_{\mathcal{F},\delta}(\X) =  \left({\displaystyle\bigcap_{\substack{C\in\mathcal{F}\\dist(\X,C)\leq\delta \\ dim(C)=n}} C}\right)^o.
\end{eqnarray}

\begin{figure}[t!]
\centering
\includegraphics[scale=0.35]{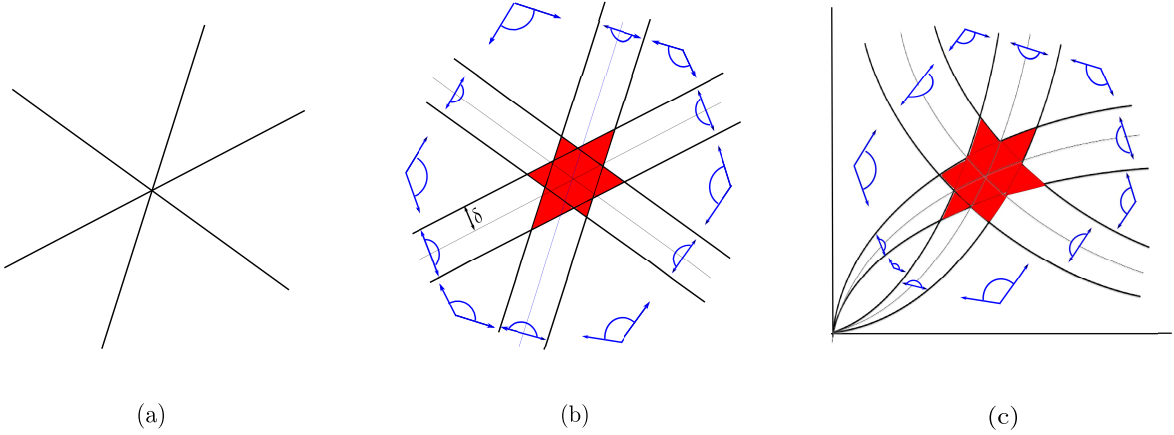}
\caption{\small (a) Hyperplane-generated polyhedral fan. This fan contains $13$ cones: six maximal cones, six cones of dimension one and one cone of dimension zero. (b) The blue cones represent the right-hand side of the toric differential inclusion corresponding to the hyperplane-generated polyhedral fan given in (a) for points $\X \! = \! \log\x$ away from the origin. For points near the origin that are demarcated by the red region, the right-hand side of the toric differential inclusion is the cone $\mathbb{R}^2$. (c) Partitioning the positive orthant $\mathbb{R}^n_{\geq 0}$ by exponentiating the lines in (b). The blue cones are \emph{the same} as in (b).}
\label{fig:toric_differential_inclusion}
\end{figure}

\noindent Figure~\ref{fig:toric_differential_inclusion} depicts the right-hand side of a toric differential inclusion, i.e., $F_{\mathcal{F},\delta}(\X)$ for a hyperplane-generated polyhedral fan.  

\section{Embedding endotactic dynamical systems into toric differential inclusions}\label{sec:embed_endotactic_TDI}

The following definition makes precise the notion of \textit{embedding} a dynamical system into a differential inclusion.

\begin{definition}
We say that the dynamical system $\frac{d\x}{dt}=f(\x,t)$ is \textit{embedded} into the differential inclusion $\frac{d\x}{dt}\in F(\x)$ if $f(\x,t)\in F(\x)$ for all $\x$ and $t$.
\end{definition}

\noindent The goal of this section is to show that any variable-$k$ endotactic dynamical system can be embedded into a toric differential inclusion. We sketch the main idea of the proof. Consider an endotactic E-graph that generates this variable-$k$ endotactic dynamical system. Note that different monomials $\x^{\s_i}$ in (\ref{eq:mass_action_dynamics}) have different magnitudes in different parts of the state space. In particular, we compare the magnitude of monomials when $\x$ is restricted to a well-chosen region given by a fixed projection of the E-graph. Since the one-dimensional projections of this E-graph are endotactic, there exists an edge corresponding to a dominant monomial that points ``inwards". As a consequence, the entire dynamics points inwards. Repeating this procedure for several one-dimensional projections, we conclude that the dynamics of the E-graph points exactly in the cone corresponding to the right-hand side of a toric differential inclusion at that point. Theorem~\ref{thm:embed_endotactic_tdi} makes these ideas precise.

\begin{theorem}\label{thm:embed_endotactic_tdi}
Any variable-$k$ endotactic dynamical system can be embedded into a toric differential inclusion.
\end{theorem}

\begin{proof}
Consider a variable-$k$ dynamical system generated by an endotactic E-graph $\GG=(V,E)$. Let $|E|=r$. Then our dynamical system can be written as\footnote{Note that $\s_1,\s_2,...,\s_r$ may not be all distinct.}
\begin{eqnarray}\label{eq:reaction_dynamics}
\frac{d\x}{dt}=\displaystyle\sum_{i=1}^r k_i(t)\x^{\s_i}(\s'_i-\s_i),
\end{eqnarray}
where $\x\in\RR^n_{>0}$ and $\epsilon\leq k_i(t)\leq\frac{1}{\epsilon}$ for some $\epsilon>0$. If $V_S$ is the set of source vertices in $\GG$, consider the set of hyperplanes $\mathcal{H}=\{(\s_i-\s_j)^{\perp}\mid \s_i\neq \s_j\in V_S\}$. Let $\mathcal{F}_{\mathcal{H}}$ denote the complete polyhedral fan generated by the hyperplanes $\mathcal{H}$. Let 
\begin{eqnarray}
\mathcal{J}=\text{span}\{\s_i-\s_j\mid \s_i,\s_j\in V_S\}.
\end{eqnarray}

\textit{Case 1}. Consider first the case when $\mathcal{J}=\RR^n$. We will show that any cone in $\mathcal{F}_{\mathcal{H}}$ is a pointed cone. Indeed, if there exists a cone $C_0\in\mathcal{F}_{\mathcal{H}}$ that is not pointed, then it contains a line $l$ passing through the origin. We will prove that $l$ is contained in every hyperplane of $\mathcal{H}$. For contradiction, assume that there is a hyperplane $H_0\in\mathcal{H}$ that does not contain $l$. Since $l$ is contained in $C_0$, it is contained in a maximal cone $C_0^n\in\mathcal{F}_{\mathcal{H}}$. Then, the maximal cone $C_0^n$ intersects the interiors of both the half-spaces generated by the hyperplane $H_0$. This contradicts the fact that, according to Definition~\ref{def:hyperplane-generated-fan}, the maximal cones of $\mathcal{F}_{\mathcal{H}}$ are minimal $n$-dimensional intersections of half-spaces generated by hyperplanes containing the origin. Therefore, $l$ is contained in every hyperplane of $\mathcal{H}$. This implies that $l$ is orthogonal to every vector $\s_i-\s_j$ for $\s_i\neq \s_j\in V_S$, contradicting the fact that $\mathcal{J}=\mathbb{R}^n$. 

Let 
\begin{eqnarray}
\mathcal{R}=\{C\in\mathcal{F}_{\mathcal{H}}\mid dim(C)=1\}
\end{eqnarray}
denote the set of one-dimensional cones of the fan $\mathcal{F}_{\mathcal{H}}$ and let $\hat{\mathcal{R}}$ denote the set of unit vectors of the cones in $\mathcal{R}$. Let $\mathcal{T}_{\mathcal{F}_{\mathcal{H}},\delta}$ denote the toric differential inclusion generated by $\mathcal{F}_{\mathcal{H}}$ and $\delta$. We will show that if we choose

\begin{eqnarray}
\delta=\frac{1}{\displaystyle\min_{\s_i\neq \s_j\in V_S}||\s_j-\s_i||}\log{\left(\frac{\displaystyle\sum_{i=1}^r ||\s'_i-\s_i||}{K_0}\right)},\end{eqnarray}
where\footnote{Note that $K_0$ is well-defined because $\mathcal{F}_{\mathcal{H}}$ is a hyperplane-generated polyhedral fan and $\text{span}\{\hat{\bn}\mid\hat{\bn}\in\hat{\mathcal{R}}\}=\mathbb{R}^n$, which implies that for any $\s_i\rightarrow\s'_i\in E$, there exists $\hat{\bn}\in\hat{\mathcal{R}}$ such that $(\s'_i-\s_i)\cdot\hat{\bn}<0$.}

\begin{eqnarray}
K_0=\epsilon^2\displaystyle\min\{-(\s'_i-\s_i)\cdot\hat{\bn}\mid \s_i\rightarrow\s'_i\in E, \hat{\bn}\in\hat{\mathcal{R}}\,\text{such that}\,(\s'_i-\s_i)\cdot\hat{\bn}<0\,\text{and}\,||\hat{\bn}||=1\},
\end{eqnarray}
then (\ref{eq:reaction_dynamics}) can be embedded into $\mathcal{T}_{\mathcal{F}_{\mathcal{H}},\delta}$.

Let us fix some arbitrary $\X=\log(\x)$. Recall from Equation (\ref{eq:useful_toric}) that the right-hand side of the toric differential inclusion $RHS(\mathcal{T}_{\mathcal{F}_{\mathcal{H}},\delta}(\X))$ is $\mathcal{P}^o$, where

\begin{eqnarray}
\mathcal{P}=\displaystyle\mathop{\bigcap}_{\substack{C\in\mathcal{F}_{\mathcal{H}}\\dist(\X,C)\leq\delta}}C.
\end{eqnarray}

\noindent If $dim(\mathcal{P})=0$, then $\mathcal{P}^o=\mathbb{R}^n$, so
\begin{eqnarray}
\frac{d\x}{dt}\bigg|_{\x=e^{\X}}\in RHS(\mathcal{T}_{\mathcal{F}_{\mathcal{H}},\delta}(\X)),
\end{eqnarray}
as desired. If $ dim(\mathcal{P})\geq 1$, we represent the cone $\mathcal{P}$ in terms of its one-dimensional generators, i.e.,
\begin{eqnarray} 
\mathcal{P}=\displaystyle\sum_{i=1}^{k} C_i,
\end{eqnarray}
\noindent for some $C_1,C_2,...,C_k\in\mathcal{R}$. Then, the right-hand side of the toric differential inclusion at $\X$ is given by

\begin{eqnarray}\label{eq:toric_differential_inclusion}
RHS(\mathcal{T}_{\mathcal{F}_{\mathcal{H}},\delta}(\X))= \left(\displaystyle\sum_{i=1}^{k} C_i\right)^o=\displaystyle\mathop{\bigcap}_{i=1}^{k}C^o_i. 
\end{eqnarray}

\noindent Choose a one-dimensional generator $C_i$ of $\mathcal{P}$ and consider the vector $\w\in C_i$ such that $||\w||=1$. Define 
\begin{eqnarray}
\sigma(\w)=\{\bM\in\mathbb{R}^n\mid dist(\bM,C)>\delta\ \text{for all}\ C\in\mathcal{F}_{\mathcal{H}}\ \text{such that}\ \w\notin C\}.
\end{eqnarray}

\begin{figure}[t!]
\centering
\includegraphics[scale=0.32]{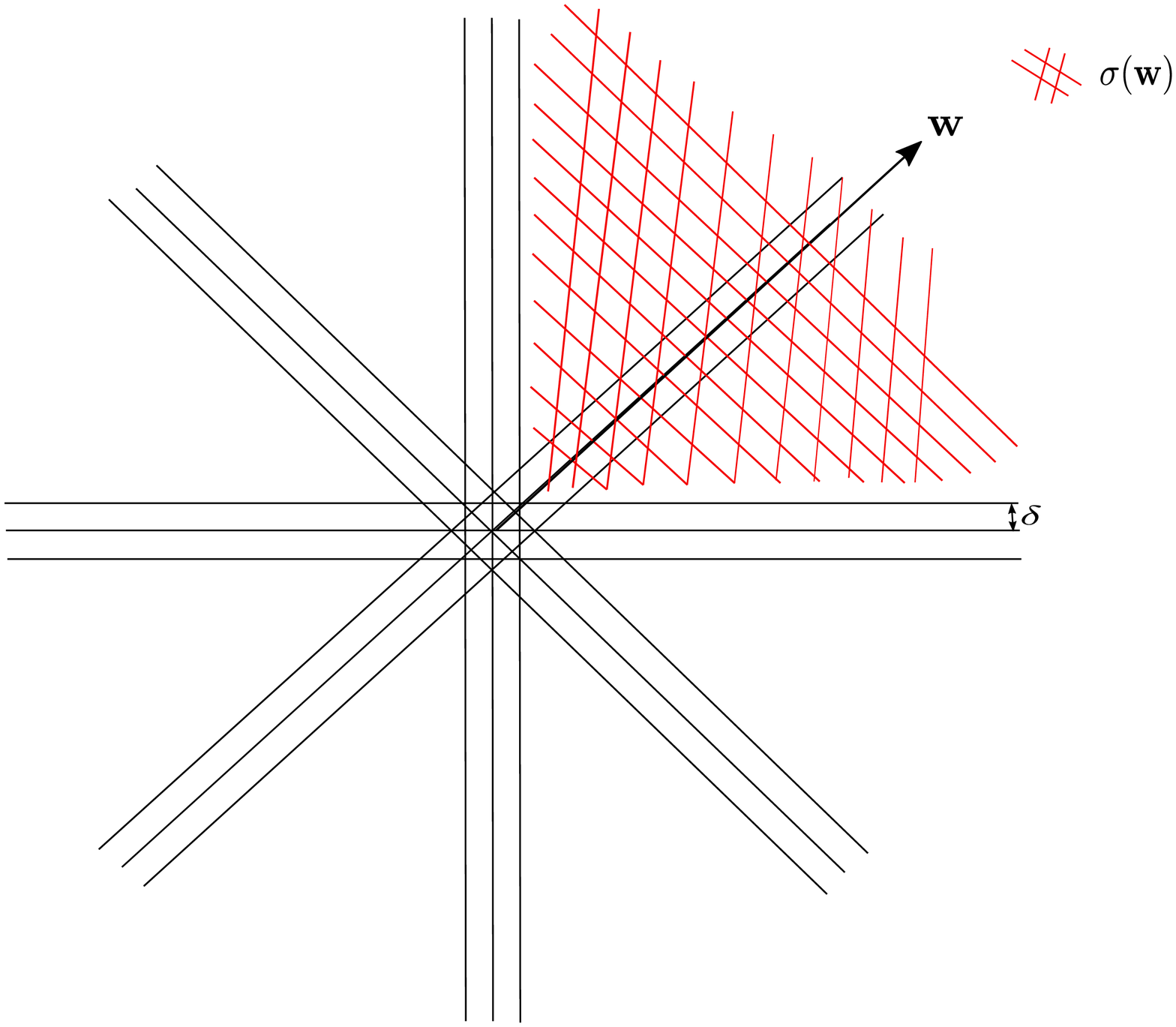}
\caption{\small Figure illustrating $\sigma(\w)$ in a hyperplane-generated polyhedral fan.}
\label{fig:strict_inequality_monomials}
\end{figure}

\noindent Refer to Figure~\ref{fig:strict_inequality_monomials} for an illustration of $\sigma(\w)$. Note that $\X\in\sigma(\w)$; indeed if $\X\notin\sigma(\w)$, then there exists a cone $\tilde{C}\in\mathcal{F}_{\mathcal{H}}$ not containing $\w$ such that $dist(\X,\tilde{C})\leq\delta$. Since 

\begin{eqnarray}
\w\in\mathcal{P}=\displaystyle\mathop{\bigcap}_{\substack{C\in\mathcal{F}_{\mathcal{H}}\\dist(\X,C)\leq\delta}}C,
\end{eqnarray}

\noindent this implies that $\w\in\tilde{C}$, a contradiction. 

Let $\bv_1,\bv_2,..., \bv_r$ be a reordering of $\s_1,\s_2,...,\s_r$ such that $(\bv_{i+1}-\bv_i)\cdot \w\leq 0 \, \text{for}\, 1\leq i\leq r-1$.\footnote{This reordering of source vertices $\s_1,\s_2,...,\s_r$ will be shown to imply important inequalities between monomials $\x^{\bv_i}$, when evaluated at $\x=e^{\X}$; see Equation (\ref{eq:imp_inequalities}).} Note that, since $\GG$ is endotactic, we cannot have $(\bv'_1-\bv_1)\cdot \w>0$, because then $(\bv'_1-\bv_1)\cdot(-\w)<0$, and this together with Definition~\ref{def:endotactic} would contradict the way we defined $\bv_1,\bv_2,...,\bv_r$. Therefore, we have $(\bv'_1-\bv_1)\cdot\w\leq 0$. Now iterate through vertices $\bv_1,\bv_2,...,\bv_r$ in the increasing order of their index. Let $\bv_j$ be the first vertex such that there exists an edge $\bv_j\rightarrow\bv'_j\in E$ such that $(\bv'_j-\bv_j)\cdot\w<0$. If no such $\bv_j$ exists, then since $\GG$ is endotactic, all edge vectors of $\GG$ are in the hyperplane given by the boundary of $C^o_i$ and therefore the right-hand side of (\ref{eq:reaction_dynamics}) lies in the half-space $C^o_i$. Else, we will show that the right-hand side of (\ref{eq:reaction_dynamics}) still lies in the half-space $C^o_i$. Let $\bv_{j+1},\bv_{j+2}... \bv_{j+s-1}$ be vertices such that $\w\cdot \bv_j = \w\cdot \bv_{j+1}= ... = \w\cdot \bv_{j+s-1}$ for some $s\geq 1$. Then, since $\GG$ is endotactic, the edge vectors with source vertices $\bv_1,\bv_2,...,\bv_{j+s-1}$ all lie in $C^o_i$. Rewriting (\ref{eq:reaction_dynamics}) in terms of the reordered vertices at $\x=e^{\X}$, we get\footnote{Note that $k'_i(t)$ refers to the rate constant of the edge $\bv_i\rightarrow\bv'_i$, where the source vertex $\bv_i$ is obtained by reordering the vertices $\s_1,\s_2,...\s_r$. Note also that, in this proof, the notation $\frac{d\x}{dt}\bigg|_{\x=e^{\X}}$ refers to the RHS of Equation (\ref{eq:reaction_dynamics}) evaluated at $\x=e^{\X}$.}

\begin{eqnarray}\label{eq:permuted_reaction_dynamics}
\begin{aligned}
\frac{d\x}{dt}\bigg|_{\x=e^{\X}} &= \displaystyle\sum_{i=1}^r k'_i(t)\x^{\bv_i}(\bv'_i-\bv_i) \\
			  &= \displaystyle\sum_{i=1}^{j-1} k'_i(t)\x^{\bv_i}(\bv'_i-\bv_i) + k'_j(t)\x^{\bv_j}(\bv'_j-\bv_j)  + \displaystyle\sum_{i=j+1}^{j+s-1}k'_i(t)\x^{\bv_i}(\bv'_i-\bv_i)\\
			  & + \displaystyle\sum_{i=j+s}^{r} k'_i(t)\x^{\bv_i}(\bv'_i-\bv_i).
\end{aligned}
\end{eqnarray}

\noindent Note that we have the following:
\begin{eqnarray}
\begin{split}
\displaystyle\sum_{i=1}^{j-1} k'_i(t)(\bv'_i-\bv_i)\cdot \w &=0, \\
k'_j(t)\x^{\bv_j}(\bv'_j-\bv_j)\cdot \w & < 0,\ \text{and} \\
\displaystyle\sum_{i=j+1}^{j+s-1}k'_i(t)\x^{\bv_i}(\bv'_i-\bv_i) \cdot \w &\leq 0. 
\end{split}
\end{eqnarray}

\noindent To show that the right-hand side of (\ref{eq:permuted_reaction_dynamics}) belongs to $C^o_i$, it therefore suffices to show that

\begin{eqnarray}
\left[k'_j(t)\x^{\bv_j}(\bv'_j-\bv_j) + \displaystyle\sum_{i=j+s}^{r} k'_i(t)\x^{\bv_i}(\bv'_i-\bv_i)\right]\cdot \w \leq 0.
\end{eqnarray}

\noindent Note that since $\w\cdot(\bv_j -\bv_l) >0$ for $j+s\leq l \leq r$, it follows that $\w$ does not lie on any of the hyperplanes orthogonal to the vectors $\bv_j-\bv_l$ and therefore $\w$ does not belong to any of the cones contained in these hyperplanes. Since $\X\in \sigma(\w)$, we have $dist(\X,C)>\delta$ for all cones $C$ in $\mathcal{F}_\mathcal{H}$ not containing $\w$. In particular, we have $dist(\X,C)>\delta$ for cones $C\in\mathcal{F}_\mathcal{H}$ that are contained in the hyperplanes mentioned above. Since these hyperplanes are a union of such cones, we have $dist(\X,H)>\delta$ for each hyperplane $H$ orthogonal to the vectors $\bv_j-\bv_l$. This implies that
\begin{eqnarray}
\X\cdot\frac{\bv_j-\bv_l}{||\bv_j-\bv_l||} > \delta = \frac{1}{\displaystyle\min_{\s_i\neq \s_j\in V_S}||\s_j-\s_i||}\log{\frac{\displaystyle\sum_{i=1}^r ||\s'_i-\s_i||}{K_0}}\geq \frac{1}{||\bv_j-\bv_l||}\log{\frac{\displaystyle\sum_{i=1}^r ||\bv'_i-\bv_i||}{-\epsilon^2(\bv'_j-\bv_j)\cdot \w}}.
\end{eqnarray}
Therefore, by exponentiating both sides, we obtain
\begin{eqnarray}\label{eq:imp_inequalities}
\x^{\bv_l}< \frac{-\epsilon^2(\bv'_j-\bv_j)\cdot \w}{\displaystyle\sum_{i=1}^r||\bv'_i-\bv_i||}\x^{\bv_j}\,\text{for all}\, j+s\leq l\leq r.
\end{eqnarray}

\noindent Now, by using $k'_j(t)\geq\epsilon$ and $k'_i(t)\leq\frac{1}{\epsilon}$ for $j+s\leq i\leq r$, we have:

\begin{eqnarray}
\begin{aligned}
\left[k'_j(t)\x^{\bv_j}(\bv'_j-\bv_j)+\displaystyle\sum_{i=j+s}^{r}k'_i(t)\x^{\bv_i}(\bv'_i-\bv_i)\right]\cdot \w & < k'_j(t)\x^{\bv_j}(\bv'_j-\bv_j)\cdot \w - \\
                                       &\frac{\epsilon^2(\bv'_j-\bv_j)\cdot \w}{\displaystyle\sum_{i=1}^r||\bv'_i-\bv_i||}\x^{\bv_j}\displaystyle\sum_{i=j+s}^{r}k'_i(t)(\bv'_i-\bv_i)\cdot \w\\
& \leq \epsilon\x^{\bv_j}\Bigg[(\bv'_j-\bv_j)\cdot\w- \\
&(\bv'_j-\bv_j)\cdot \w\frac{\displaystyle\sum_{i=1}^{r}||\bv'_i-\bv_i||}{\displaystyle\sum_{i=1}^r ||\bv'_i-\bv_i||}\Bigg] =0
\end{aligned}
\end{eqnarray}	

\noindent This implies that $\frac{d\x}{dt}\bigg|_{\x=e^{\X}}\in C^o_i$. Repeating this for each of the one-dimensional generators of $\mathcal{P}$ given by $C_1,C_2,... ,C_{k}$, we get
\begin{eqnarray}\label{eq:particular_inclusion}
\frac{d\x}{dt}\bigg|_{\x=e^{\X}}\in\displaystyle\mathop{\bigcap}_{i=1}^{k}C^o_i.
\end{eqnarray}
Since the choice of point $\X$ was arbitrary, using (\ref{eq:toric_differential_inclusion}) we get 
\begin{eqnarray}
\frac{d\x}{dt}\bigg|_{\x=e^{\X}}\in RHS(\mathcal{T}_{\mathcal{F}_{\mathcal{H}},\delta}(\X))
\end{eqnarray}
for all $\X\in\mathbb{R}^n$, as desired.

\bigskip\bigskip
\textit{Case 2}. Consider now the case when $\mathcal{J}\neq\RR^n$. In this case, the cones in the polyhedral fan $\mathcal{F}_{\mathcal{H}}$ are not pointed, but note that $C\cap\mathcal{J}$ is a pointed cone for all $C\in\mathcal{F}_{\mathcal{H}}$. Moreover, the orthogonal complement of $\mathcal{J}$ denoted by $\mathcal{J}^{\perp}$ is contained in all hyperplanes $(\s_i-\s_j)^{\perp}$. Further, since $\mathcal{J}$ is a linear subspace we have $\mathcal{J}^o=\mathcal{J}^{\perp}$. Therefore, we get
\begin{eqnarray}\label{eq:non_pointed_component}
C^o\subseteq\mathcal{J}\ \text{for all}\ C\in\mathcal{F}_{\mathcal{H}}.
\end{eqnarray}
Since $\GG$ is endotactic, we have
\begin{eqnarray}
\text{span}\{\s'_i-\s_i\mid \s_i\rightarrow\s'_i\in E\}\subseteq\mathcal{J}.
\end{eqnarray}
\noindent Now, proceed as in the previous case with 
\begin{eqnarray}
\mathcal{R}=\{C\cap\mathcal{J}\mid C\in\mathcal{F}_{\mathcal{H}}\ \text{and}\ dim(C\cap\mathcal{J})=1\}.
\end{eqnarray}
Fix $\X\in\mathbb{R}^n$. Let 
\begin{eqnarray}
\mathcal{P}=\displaystyle\mathop{\bigcap}_{\substack{C\in\mathcal{F}_{\mathcal{H}}\\dist(\X,C)\leq\delta}}C,
\end{eqnarray}
\noindent as in the previous case. We represent the pointed cone $\mathcal{P}\cap\mathcal{J}$ in terms of its one-dimensional generators, i.e.,
\begin{eqnarray}
\mathcal{P}\cap\mathcal{J}=\displaystyle\sum_{i=1}^{k} C_i.
\end{eqnarray}
Therefore, we get\footnote{For any two cones $C_1,C_2$, their Minkowksi sum is given by 
\begin{eqnarray*}
C_1 + C_2 = \{\x_1 + \x_2\in\mathbb{R}^n | \x_1\in C_1\ \text{and}\ \x_2\in C_2\}.
\end{eqnarray*}
}
\begin{eqnarray}
(\mathcal{P}\cap\mathcal{J})^o\cap\mathcal{J} = (\mathcal{P}^o + \mathcal{J}^o)\cap\mathcal{J}  = (\mathcal{P}^o + \mathcal{J}^{\perp})\cap\mathcal{J}.
\end{eqnarray}
Let $\bu\in (\mathcal{P}^o + \mathcal{J}^{\perp})\cap\mathcal{J}$. This means $\bu\in\mathcal{P}^o + \mathcal{J}^{\perp}$ and $\bu\in\mathcal{J}$. Using Equation (\ref{eq:non_pointed_component}), we have
\begin{eqnarray}
\bu = \hat{\bu} + \tilde{\bu},
\end{eqnarray}
where $\hat{\bu}\in\mathcal{J}$ and $\tilde{\bu}\in\mathcal{J}^{\perp}$. This implies that $\tilde{\bu}=0$ and using Equation (\ref{eq:non_pointed_component}) again, we get
\begin{eqnarray}
(\mathcal{P}^o + \mathcal{J}^{\perp})\cap\mathcal{J} = \mathcal{P}^o \cap\mathcal{J} = \mathcal{P}^o.
\end{eqnarray}
Therefore, 
\begin{eqnarray}\label{eq:modified_toric_differential_inclusion}
\mathcal{P}^o = (\mathcal{P}\cap\mathcal{J})^o\cap\mathcal{J} =\left(\displaystyle\sum_{i=1}^{k} C_i\right)^o\cap\mathcal{J}=\displaystyle\mathop{\bigcap}_{i=1}^{k}C^o_i\cap\mathcal{J}. 
\end{eqnarray}
By proceeding as in the previous case, we obtain
\begin{eqnarray}
\frac{d\x}{dt}\bigg|_{\x=e^{\X}}\in\displaystyle\mathop{\bigcap}_{i=1}^{k}C^o_i.
\end{eqnarray}
Note that we also have 
\begin{eqnarray}
\frac{d\x}{dt}\bigg|_{\x=e^{\X}}\in\text{span}\{\s'_i-\s_i\mid \s_i\rightarrow\s'_i\in E\}\subseteq\mathcal{J}.
\end{eqnarray}
\noindent Therefore, we obtain
\begin{eqnarray}\label{eq:second_last}
\frac{d\x}{dt}\bigg|_{\x=e^{\X}}\in\displaystyle\mathop{\bigcap}_{i=1}^{k}C^o_i\cap\mathcal{J}.
\end{eqnarray}
Recall that $RHS(\mathcal{T}_{\mathcal{F}_{\mathcal{H}},\delta}(\X))= \mathcal{P}^o$. Since the choice of point $\X$ was arbitrary, using Equations (\ref{eq:modified_toric_differential_inclusion}) and (\ref{eq:second_last}), we get
\begin{eqnarray}
\frac{d\x}{dt}\bigg|_{\x=e^{\X}}\in RHS(\mathcal{T}_{\mathcal{F}_{\mathcal{H}},\delta}(\X)),
\end{eqnarray}
for all $\X\in\mathbb{R}^n$, as desired.

\end{proof}

\section{A converse of Theorem~\ref{thm:embed_endotactic_tdi}}\label{sec:converse_endotactic_TDI}

In this section, we show a converse of Theorem~\ref{thm:embed_endotactic_tdi}. More precisely, we prove that if an E-graph $\GG=(V,E)$ is not endotactic, then there exists a variable-$k$ dynamical system generated by $\GG$ that cannot be embedded into any toric differential inclusion. The key observation is that since $\GG$ is not endotactic, there exists a one-dimensional projection of $\GG$ along which there is a ``bad" edge corresponding to the most dominant source monomial that ``points outwards". There might be other ``good" edges corresponding to the most dominant source monomials that do point inwards. But we can choose the rate constants $k_i(t)$ of these ``good" edges arbitrarily small by choosing a sufficiently small $\epsilon$. This forces the dynamics to lie outside the cone corresponding to the right-hand side of the toric differential inclusion and the desired result follows.

\begin{theorem}\label{thm:converse_endotactic}
Consider an E-graph $\GG=(V,E)$ such that any variable-$k$ dynamical system generated by $\GG$ can be embedded into a toric differential inclusion. Then $\GG$ is endotactic.
\end{theorem}

\begin{proof}
We shall show that if $\GG=(V,E)$ is not an endotactic E-graph, then there exists $\epsilon>0$ such that a variable-$k$ system generated by $\GG$ cannot be embedded into a toric differential inclusion $\mathcal{T}_{\mathcal{F},\delta}(\X)$ for any complete polyhderal fan $\mathcal{F}$ and any $\delta>0$. 

Let $|E|=r$ and $\s_1\rightarrow \s'_1,\s_2\rightarrow\s'_2,...,\s_r\rightarrow\s'_r$ denote the edges in $E$. If $\GG$ is not an endotactic E-graph, then there exists a $\w\in\RR^m$ and $\s_i\rightarrow\s'_i\in E$ such that $\w\cdot (\s'_i-\s_i)<0$ and $\w\cdot (\s_j-\s_i)\geq 0$ for every $(\s_j,\s'_j)\in E$ satisfying $\w\cdot \s_j\neq \w\cdot\s'_j$. Let $\w'=-\w$ and

\begin{eqnarray}
\epsilon=\sqrt{\frac{\min\{(\s'_i-\s_i)\cdot \w'\mid  \s_i\rightarrow\s'_i\in E,(\s'_i-\s_i)\cdot \w'>0\}}{2||\w'||\displaystyle\sum_{i=1}^{r}||\s'_i-\s_i||}}.
\end{eqnarray}

Let $V_S$ denote the set of source vertices in $\GG=(V,E)$. Let $\bv_1,\bv_2,..., \bv_r$ be a reordering of the vertices $\s_1,\s_2,...,\s_r$ such that $(\bv_{i+1}-\bv_i)\cdot \w'\leq 0 \,\text{for} \, 1\leq i\leq r-1$. Now iterate through vertices $\bv_1,\bv_2,...,\bv_r$ in the increasing order of their index. Let $\bv_j$ be the first vertex such that there exists an edge $\w'\cdot(\bv'_j - \bv_j)>0$. Due to our choice of $\w$, such a $\bv_j$ always exists. Let $\bv_{j+1},\bv_{j+2}... \bv_{j+s-1}$ be vertices such that $\w'\cdot \bv_j = \w'\cdot \bv_{j+1}=... = \w'\cdot \bv_{j+s-1}$ for some $s\geq 1$. Define 

\begin{eqnarray}
ray(\w')=\{\X\in\RR^n\mid \X=\lambda \w'\,\text{for some}\,\lambda\in\RR_{\geq 0}\}.
\end{eqnarray}
A variable-$k$ dynamical system generated by $\GG$ can be written as
\begin{eqnarray}\label{eq:converse_reaction_dynamics}
\frac{d\x}{dt}=\displaystyle\sum_{i=1}^r k_i(t)\x^{\s_i}(\s'_i-\s_i),
\end{eqnarray}
where $\x\in\RR^n_{>0}$ and $\epsilon\leq k_i(t)\leq\frac{1}{\epsilon}$. Consider a toric differential inclusion $\mathcal{T}_{\mathcal{F},\delta}(\X)$ generated by a polyhedral fan $\mathcal{F}$ and $\delta>0$. We shall show that there exists $\X_0\in ray(\w')$ such that $\frac{d\x}{dt}\bigg|_{\x=e^{\X_0}}\notin F_{\mathcal{F},\delta}(\X_0)$. Towards this, we will prove that there exist $\X_0\in ray(\w')$ such that (see Figure~\ref{fig:not_embedded_toric}):
\begin{eqnarray}
F_{\mathcal{F},\delta}(\X_0)\subseteq {\{ray(\w')\}}^o\ \text{and}\ \frac{d\x}{dt}\bigg|_{\x=e^{\X_0}}\notin ray(\w')^o.
\end{eqnarray}

We first show that $F_{\mathcal{F},\delta}(\X_0)\subseteq {\{ray(\w')\}}^o$. Let us denote $C_{\w'}$ to be the intersection of all the cones in the polyhedral fan containing $\w'$, i.e., $C_{\w'}=\displaystyle\mathop{\bigcap}_{\substack{C\in\mathcal{F}\\ \w'\in C}}C$. Now choose $\X_0\in ray(\w')$ sufficiently far away from the origin so that 
\begin{eqnarray}
dist(\X_0,C)>\max\left(\delta,\,\frac{1}{\displaystyle\min_{\bv_i\neq \bv_j\in V_S}||\bv_j-\bv_i||}\log\left(\frac{2||\w'||\displaystyle\sum_{i=1}^r ||\bv'_i-\bv_i||}{\epsilon^2(\bv'_j-\bv_j)\cdot \w'}\right)\right)
\end{eqnarray}

\noindent for every cone $C$ not\footnote{Such an $\X_0$ does exist because the distance between $ray(\w')$ and each cone $C$ not containing $w'$ within the unit sphere $S^n$ is strictly positive.} containing $\w'$. Note that 

\begin{eqnarray}
F_{\mathcal{F},\delta}(\X_0)=\left({\displaystyle\bigcap_{\substack{C\in\mathcal{F}\\dist(\X_0,C)\leq\delta}} C}\right)^o = \left(
\left({\displaystyle\bigcap_{\substack{C\in\mathcal{F}\\dist(\X_0,C)\leq\delta \\ w'\in C}} C}
\right)
\bigcap
\left(
{\displaystyle\bigcap_{\substack{C\in\mathcal{F}\\dist(\X_0,C)\leq\delta \\ w'\notin C}} C}\right) 
\right)^o.
\end{eqnarray}

But since $dist(\X_0,C)> \delta$ for every cone $C\in\mathcal{F}$ not containing $w'$, we have that the right-hand side of the toric differential inclusion at $\X_0$ is 

\begin{eqnarray}
F_{\mathcal{F},\delta}(\X_0)=\left(\displaystyle\mathop{\bigcap}_{\substack{C\in\mathcal{F}\\ \w'\in C}}C\right)^o=C^o_{\w'}.
\end{eqnarray}

\noindent Since $ray(\w')\subseteq C_{\w'}$, we get
\begin{eqnarray}
F_{\mathcal{F},\delta}(\X_0)=C^o_{\w'}\subseteq {\{ray(\w')\}}^o,
\end{eqnarray}
as desired. 

\begin{figure}[t!]
\centering
\includegraphics[scale=0.4]{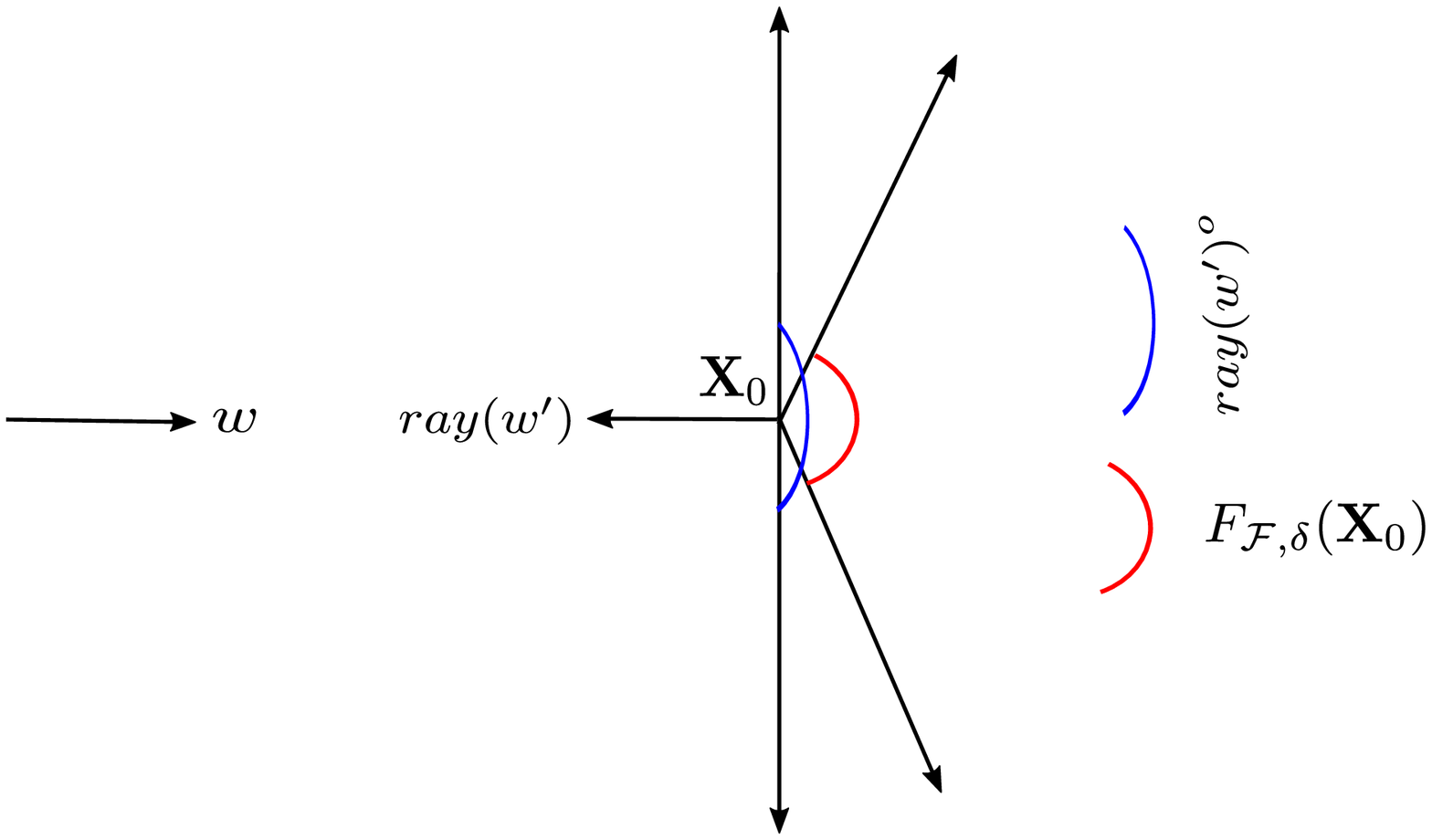}
\caption{\small Illustration of the fact that to show $\frac{d\x}{dt}\bigg|_{\x=e^{\X_0}}\notin F_{\mathcal{F},\delta}(\X_0)$, it suffices to show $\frac{d\x}{dt}\bigg|_{\x=e^{\X_0}}\notin {ray(\w')}^o$.
}
\label{fig:not_embedded_toric}     
\end{figure} 
   
\bigskip We now show that $\frac{d\x}{dt}\bigg|_{\x=e^{\X_0}}\notin {\{ray(\w')\}}^o$ or equivalently $\frac{d\x}{dt}\bigg|_{\x=e^{\X_0}}\cdot \w'>0$. Define $\x_0=e^{\X_0}$. Note that since $\X_0\in ray(\w')$, we have 
\begin{eqnarray}
\x^{\bv_1}_0 \geq\x^{\bv_2}_0 \geq ... > \x^{\bv_j}_0=\x^{\bv_{j+1}}_0 ... =\x^{\bv_{j+s-1}}_0 \geq \x^{\bv_{j+s}}_0 \geq ... \geq\x^{\bv_{r}}_0.
\end{eqnarray}
Rewriting (\ref{eq:converse_reaction_dynamics}) in terms of the reordered vertices at $\x_0$, we get

\begin{eqnarray}\label{eq:permuted_reaction_dynamics_converse}
\begin{aligned}
\frac{d\x}{dt}\bigg|_{\x=e^{\X_0}} &= \displaystyle\sum_{i=1}^r k'_i(t)\x_0^{\bv_i}(\bv'_i-\bv_i) \\
			    &= \displaystyle\sum_{i=1}^{j-1} k'_i(t)\x_0^{\bv_i}(\bv'_i-\bv_i) + k'_j(t)\x_0^{\bv_j}(\bv'_j-\bv_j)  + \x_0^{\bv_j}\displaystyle\sum_{i=j+1}^{j+s-1}k'_i(t)(\bv'_i-\bv_i) + \displaystyle\sum_{i=j+s}^{r} k'_i(t)\x_0^{\bv_i}(\bv'_i-\bv_i) \\
		        &= \displaystyle\sum_{i=1}^{j-1} k'_i(t)\x_0^{\bv_i}(\bv'_i-\bv_i) + \x_0^{\bv_j}\left[k'_j(t)(\bv'_j-\bv_j)  + \displaystyle\sum_{i=j+1}^{j+s-1}k'_i(t)(\bv'_i-\bv_i)\right] + \displaystyle\sum_{i=j+s}^{r} k'_i(t)\x_0^{\bv_i}(\bv'_i-\bv_i) \\ 
		        &= \displaystyle\sum_{i=1}^{j-1} k'_i(t)\x_0^{\bv_i}(\bv'_i-\bv_i) + \x_0^{\bv_j}\left[\frac{k'_j(t)}{2}(\bv'_j-\bv_j)  + \displaystyle\sum_{i=j+1}^{j+s-1}k'_i(t)(\bv'_i-\bv_i)\right] + \x_0^{\bv_j}\frac{k'_j(t)}{2}(\bv'_j-\bv_j) \\
		        &+ \displaystyle\sum_{i=j+s}^{r} k'_i(t)\x_0^{\bv_i}(\bv'_i-\bv_i).
\end{aligned}
\end{eqnarray}

\noindent Note that, due to the way we chose the index $j$, we have $(\bv'_i-\bv_i)\cdot\w'=0$ for all $1\leq i < j$, and therefore

\begin{eqnarray}
\begin{split}
\displaystyle\sum_{i=1}^{j-1} k'_i(t)\x_0^{\bv_i}(\bv'_i-\bv_i)\cdot \w'& =0\ \text{and} \\
k'_j(t)\x_0^{\bv_j}(\bv'_j-\bv_j)\cdot \w'& >0.
\end{split}
\end{eqnarray}

\noindent Then, to show that $\frac{d\x}{dt}\bigg|_{\x=e^{\X_0}}\cdot \w' >0$, it suffices to show the following
\begin{eqnarray}\label{eq:39}
\left[\frac{k'_j(t)}{2}(\bv'_j-\bv_j)  + \displaystyle\sum_{i=j+1}^{j+s-1}k'_i(t)(\bv'_i-\bv_i)\right]\cdot \w'\geq 0\ \text{and} 
\end{eqnarray}
\begin{eqnarray}\label{eq:40}
\left[\x_0^{\bv_j}\frac{k'_j(t)}{2}(\bv'_j-\bv_j) + \displaystyle\sum_{i=j+s}^{r}k'_i(t)\x_0^{\bv_i}(\bv'_i-\bv_i)\right]\cdot \w'>0.
\end{eqnarray}
We first show (\ref{eq:39}). Towards this, we have
\begin{eqnarray}
\left[\frac{k'_j(t)}{2}(\bv'_j-\bv_j)  + \displaystyle\sum_{i=j+1}^{j+s-1}k'_i(t)(\bv'_i-\bv_i)\right]\cdot \w'\geq \frac{k'_j(t)}{2}(\bv'_j-\bv_j)\cdot \w' -||\w'||\displaystyle\sum_{i=j+1}^{j+s-1}k'_i(t)||\bv'_i-\bv_i ||. 
\end{eqnarray}
\noindent We choose $k'_j(t)=\frac{1}{\epsilon}$ and $k'_i(t)=\epsilon$ for all $t\geq 0$ and for every $j+1\leq i\leq j+s-1$. Therefore, we get
\begin{eqnarray}
\begin{aligned}
\left[\frac{k'_j(t)}{2}(\bv'_j-\bv_j)  + \displaystyle\sum_{i=j+1}^{j+s-1}k'_i(t)(\bv'_i-\bv_i)\right]\cdot \w' & \geq \frac{1}{2\epsilon}(\bv'_j-\bv_j)\cdot \w' - \epsilon||\w'||\displaystyle\sum_{i=j+1}^{j+s-1}||\bv'_i-\bv_i|| \\
& \geq \frac{1}{2\epsilon}(\bv'_j-\bv_j)\cdot \w' - \epsilon||\w'||\displaystyle\sum_{i=1}^{r}||\bv'_i-\bv_i|| \\
& \geq \frac{1}{2\epsilon}\left[(\bv'_j-\bv_j)\cdot \w' - 2\epsilon^2||\w'||\displaystyle\sum_{i=1}^{r}||\bv'_i-\bv_i||\right] \\
& = \frac{1}{2\epsilon}\Bigg[(\bv'_j-\bv_j)\cdot \w' - \min\{(\s'_i-\s_i)\cdot \w' \mid \\
&  \s_i\rightarrow\s'_i\in E,(\s'_i-\s_i)\cdot \w'>0\}\Bigg]\geq 0.
\end{aligned}
\end{eqnarray}

\bigskip We now proceed to showing (\ref{eq:40}). Note that since $\w'\cdot \bv_j > \w'\cdot \bv_l$ for $j+s\leq l \leq r$, $\w'$ does not lie on any of the hyperplanes orthogonal to the vectors $\bv_j-\bv_l$. In addition, note that

\begin{eqnarray}
dist(\X_0,C)>\frac{1}{\displaystyle\min_{\bv_i\neq \bv_j\in V_S}||\bv_j-\bv_i||}\log{\frac{2||\w'||\displaystyle\sum_{i=1}^r ||\bv'_i-\bv_i||}{\epsilon^2(\bv'_j-\bv_j)\cdot \w'}}
\end{eqnarray}

\noindent for cones $C\in\mathcal{F}_\mathcal{H}$ not containing $\w'$ that are contained in the hyperplanes mentioned above. Like in the proof of Theorem~\ref{thm:embed_endotactic_tdi}, we have $dist(\X_0,H)>\delta$ for each hyperplane $H$ orthogonal to the vectors $\bv_j-\bv_l$. This means:
\begin{eqnarray}
\X_0\cdot\frac{\bv_j-\bv_l}{||\bv_j-\bv_l||} > \frac{1}{\displaystyle\min_{\bv_i\neq \bv_j\in V_S}||\bv_j-\bv_i||}\log{\frac{2||\w'||\displaystyle\sum_{i=1}^r ||\bv'_i-\bv_i||}{\epsilon^2(\bv'_j-\bv_j)\cdot \w'}} \geq \frac{1}{||\bv_j-\bv_l||}\log{\frac{2||\w'||\displaystyle\sum_{i=1}^r ||\bv'_i-\bv_i||}{\epsilon^2(\bv'_j-\bv_j)\cdot \w'}}.
\end{eqnarray}

\noindent This implies 
\begin{eqnarray}
\x_0^{\bv_l}< \frac{\epsilon^2(\bv'_j-\bv_j)\cdot \w'}{2||\w'||\displaystyle\sum_{i=1}^r||\bv'_i-\bv_i||}\x_0^{\bv_j}\ \text{for all}\ j+s\leq l\leq r.
\end{eqnarray}

\noindent We now have 

\begin{eqnarray}
\begin{aligned}
\left[\frac{k'_j(t)}{2}\x_0^{\bv_j}(\bv'_j-\bv_j)+\displaystyle\sum_{i=j+s}^{r}k'_i(t)\x_0^{\bv_i}(\bv'_i-\bv_i)\right]\cdot \w' &> \frac{k'_j(t)}{2}\x_0^{\bv_j}(\bv'_j-\bv_j)\cdot \w' -\\
& \frac{\epsilon^2(\bv'_j-\bv_j)\cdot \w'}{2\displaystyle\sum_{i=1}^r||\bv'_i-\bv_i||}\x_0^{\bv_j}\displaystyle\sum_{i=1}^{r}k'_i(t)||\bv'_i-\bv_i|| \\
& \geq\frac{\epsilon \x_0^{\bv_j}}{2}\left[(\bv'_j-\bv_j)\cdot \w'-(\bv'_j-\bv_j)\cdot \w'\right]=0.
\end{aligned}
\end{eqnarray}

\noindent This implies that $\frac{d\x_0}{dt}\notin F_{\mathcal{F},\delta}(\X_0)$. Therefore, for this choice of $\epsilon$, the variable-$k$ system discussed above cannot be embedded into any toric differential inclusion.
\end{proof}

Note that in the proof of Theorem~\ref{thm:converse_endotactic}, we have \textit{not} used the full strength of the hypothesis about any \textit{variable-$k$} system generated by $\GG$. Exactly the same proof still works under the hypothesis of \textit{fixed} but arbitrary $k$. Therefore we obtain the following.

\begin{theorem}\label{thm:stronger_converse}
Consider an E-graph $\GG=(V,E)$ such that for any positive vector of rate constants $k$, the dynamical system generated by $\GG$ and $k$ can be embedded into a toric differential inclusion. Then $\GG$ is endotactic.
\end{theorem}

In Theorem~\ref{thm:converse_endotactic}, it is essential that the hypothesis holds for any variable-$k$ dynamical system generated by $\GG$. Let us note that there exist variable-$k$ dynamical systems that are not endotactic, but can still be embedded into a toric differential inclusion: 

\begin{example}
Consider the following E-graph: $X\xrightarrow{k_1(t)}\phi, X\xrightleftharpoons[k_3(t)]{k_2(t)} 3X$. Let $\frac{1}{\sqrt{2}}<\epsilon< 1$. Note that this E-graph is not endotactic due to the edge $X\rightarrow\phi$. We will show that one can embed the dynamical system generated by this network into the toric differential inclusion $\mathcal{T}_{\mathcal{F},\delta}$ with
\begin{eqnarray}
\delta=\max\left(\log{\frac{1}{\epsilon}},\frac{1}{2}\log\left(\frac{2\epsilon}{2\epsilon^2-1}\right)\right) 
\end{eqnarray}
and $\mathcal{F}$ is the only non-trivial polyhedral fan in one-dimension, i.e., 
\begin{eqnarray}
\mathcal{F}=\{{0},\RR_{\geq 0},\RR_{\leq 0}\}.
\end{eqnarray}

Note that 
\begin{eqnarray}
\frac{dx}{dt}=-k_1(t)x+2k_2(t)x-2k_3(t)x^{3}.
\end{eqnarray}
For $x>e^{\delta}$, we have
\begin{eqnarray}
\frac{dx}{dt}\leq -\epsilon x + 2x\left(\frac{1}{\epsilon} - \epsilon x^2\right) < -\epsilon x + 2x\left(\frac{1}{\epsilon} - \epsilon e^{2\delta}\right)<-\epsilon x <0.
\end{eqnarray}
For $x<e^{-\delta}$,
\begin{eqnarray}
\frac{dx}{dt}\geq\left(2\epsilon -\frac{1}{\epsilon}\right)x - \frac{2}{\epsilon}x^{3} = x\left[\left(2\epsilon -\frac{1}{\epsilon}\right) - \frac{2}{\epsilon}x^2\right] > x\left[\left(2\epsilon -\frac{1}{\epsilon}\right) - 2e^{-2\delta}\right] >0. 
\end{eqnarray}

Therefore, we get $\frac{dx}{dt}\in F_{\mathcal{F},\delta}(X)$ for all $x$.
\end{example}

\section{Examples}

It has been shown in~\cite{craciun2019polynomial} that variable-$k$ weakly reversible dynamical systems can be embedded into  toric differential inclusions. Theorem~\ref{thm:embed_endotactic_tdi} shows that the larger class of variable-$k$ endotactic dynamical systems can be embedded into toric differential inclusions. It is notable that this embedding into toric differential inclusions also works for dynamical systems generated by endotactic E-graphs that are (in general) not weakly reversible. Such networks manifest themselves quite naturally in biological networks, as we show below.

\subsection{Circadian clock network.} This network is part of a model for the daily sleep/wake cycle in living organisms. 
For some choices of parameter values, the dynamics of this network can give rise to oscillations that repeat approximately every 24 hours. In~\cite{johnston2016computational}, Johnston, Pantea and Donnell consider several versions of a circadian clock introduced in~\cite{leloup1999chaos}, with variable-$k$ mass-action kinetics. In particular, they analyse a basic and a general type of this model, as illustrated below. 

\begin{eqnarray}
\begin{split}
\text{Basic model} &:\{ P + T \rightleftharpoons C,  C \rightarrow 0, 0 \rightleftharpoons P, 0 \rightleftharpoons T \}. \\
\text{General model} &:\{P_{n_{P}} + T_{n_{T}} \rightleftharpoons C_0, C_0\rightleftharpoons C_1\rightleftharpoons ... \rightleftharpoons C_{n_C}, \\
               & C_{n_C}\rightarrow 0, 0 \rightarrow P_0\rightleftharpoons ... \rightleftharpoons P_{n_{P}}, 0 \rightarrow T_0\rightleftharpoons ... \rightleftharpoons T_{n_{T}} \}.
\end{split}
\end{eqnarray}

\begin{figure}[h!]
\centering
\includegraphics[scale=0.43]{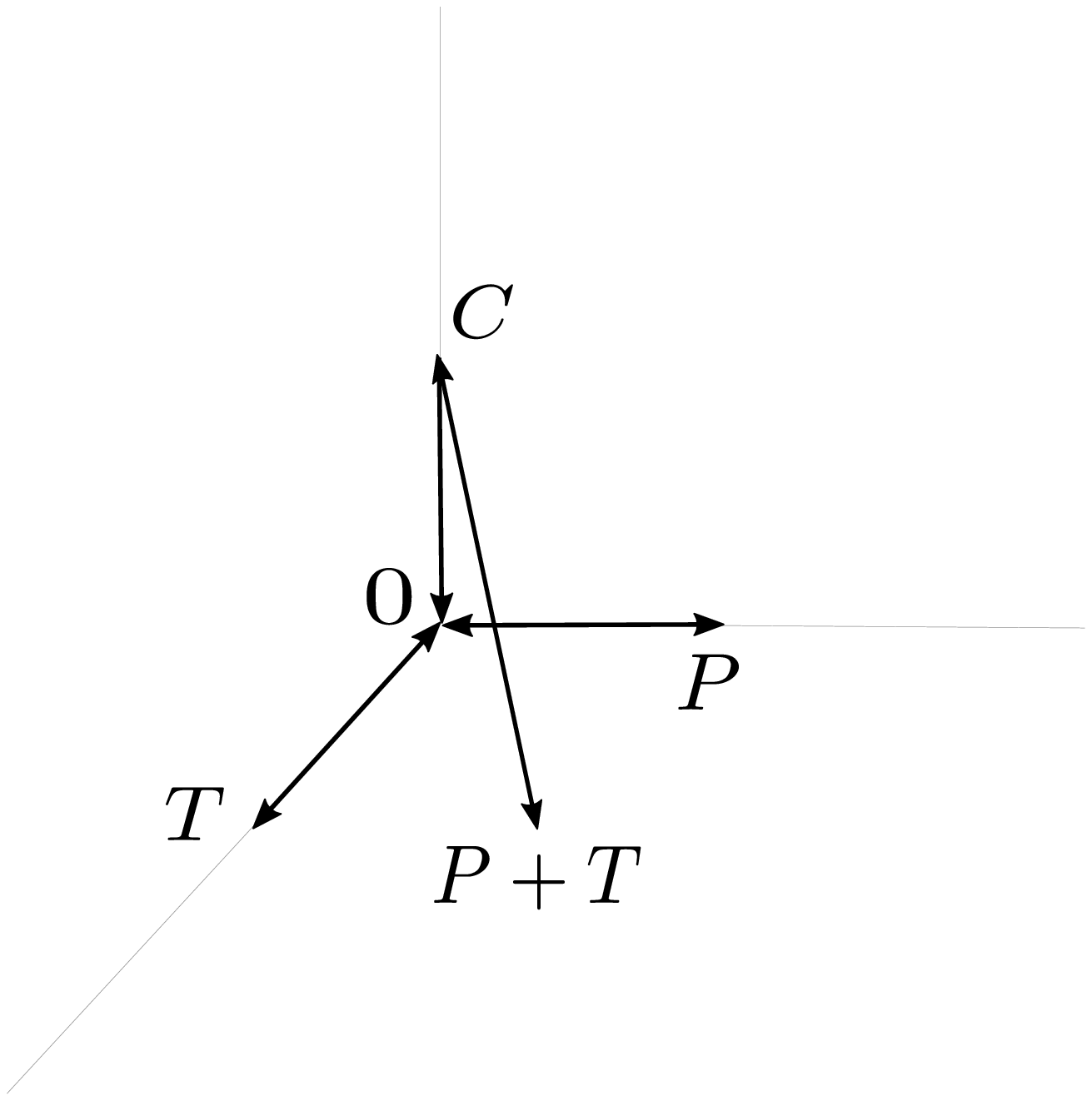}
\caption{\small E-graph depicting a basic circadian clock model.}
\label{fig:circadian_clock}
\end{figure}

Here, $P_i$ and $T_i$ denote the proteins \textit{PER} (period) and \textit{TIM} (timeless) respectively, in various phosphorylation states. $C_i$ denotes different forms of the \textit{PER-TIM} complex. Figure~\ref{fig:circadian_clock} shows the E-graph corresponding to the basic version of the clock. In~\cite{johnston2016computational}, they prove that both the basic and general versions of the circadian clock model are endotactic. Using Theorem~\ref{thm:embed_endotactic_tdi}, we can conclude that the dynamics generated by such circadian clocks can be embedded into a toric differential inclusion.

\subsection{Thomas-type models.} The Thomas-type model~\cite{johnston2016computational,murray2003mathematical} consists of an  inhibition mechanism of the substrate involving uric acid and oxygen, that is catalysed by uricase.

\begin{figure}[h!]
\centering
\includegraphics[scale=0.5]{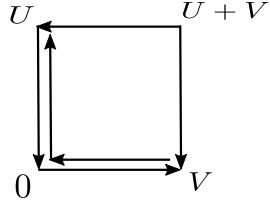}
\caption{\small Example of a Thomas-type model. Here $U$ denotes uric acid and $V$ denotes oxygen.}
\label{fig:thomas_model}
\end{figure}

Figure~\ref{fig:thomas_model} shows an example of a Thomas-type model. This network is shown to be endotactic~\cite{craciun2013persistence} and hence its dynamics can be embedded into a toric differential inclusion.

\subsection{Power law systems.} Power law systems show up very often as mathematical models of biological interaction networks, especially when trying to obtain a good fit with experimental data~\cite{savageau1969biochemical}. A wide variety of power law dynamical systems can be generated from endotactic E-graphs that are not weakly reversible.
For example, consider the following dynamical system 

\begin{eqnarray}\label{eq:power_law_dynamics}
\begin{split}
\frac{dx}{dt}  &= -k_0(t) + k_1(t)x^{-2.3} -k_2(t)y^{3.1} + 0.3k_3(t)x^{-0.5}y^{0.6}\\
\frac{dy}{dt}  &= 0.3k_1(t)x^{-2.3} -k_2(t)y^{3.1} - 0.3k_3(t)x^{-0.5}y^{0.6} \\
\frac{dz}{dt}  &= 0.8k_3(t)x^{-0.5}y^{0.6} - k_4(t)z^2.
\end{split}
\end{eqnarray}

Figure~\ref{fig:power_law_system} shows an E-graph that generates the dynamics given by Equation~\ref{eq:power_law_dynamics}. In this figure, the edge vectors for the edges labelled by the rate constants $k_0,k_1,k_2,k_3,k_4$ are \\$\begin{bmatrix}-1 \\ 0 \\ 0 \end{bmatrix}, \begin{bmatrix}1 \\ 0.3 \\ 0 \end{bmatrix}, \begin{bmatrix}-1 \\ -1 \\ 0 \end{bmatrix}, \begin{bmatrix}0.3 \\ -0.3 \\ 0.8 \end{bmatrix}, \begin{bmatrix} 0 \\ 0 \\ -1 \end{bmatrix}$ respectively. 

\begin{figure}[h!]
\centering
\includegraphics[scale=0.47]{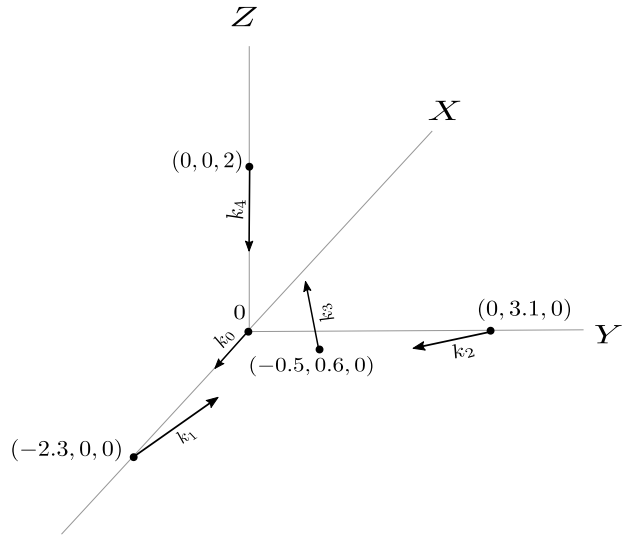}
\caption{\small E-graph depicting an endotactic power-law system.}
\label{fig:power_law_system}
\end{figure}

This E-graph is endotactic and hence the dynamical system generated by it can be embedded into a toric differential inclusion. Moreover, we could replace the terms $0.3k_3(t)x^{-0.5}y^{0.6}$ and $-0.3k_3(t)x^{-0.5}y^{0.6}$ in (\ref{eq:power_law_dynamics}) by $\pm k_5(t)x^{-0.5}y^{0.6}$ and $\pm k_6x^{-0.5}y^{0.6}$ and the corresponding E-graph is still endotactic.

\subsection{Hypercycle-type models.} Our work also has connections to certain hypercycle-type network models~\cite{eigen2012hypercycle}. More precisely, in upcoming work, we show that the dynamics of relative concentrations for certain hypercycle-type networks coincides with the dynamics of some endotactic networks. For example, consider the following hypercycle-type network:

\begin{eqnarray}\label{eq:hypercycle}
\begin{split}
X_1 + X_2 &\xrightarrow[]{k_1} X_1 + 2X_2 \\
X_2 + X_3 &\xrightarrow[]{k_2} X_2 + 2X_3\\
X_3 + X_1 &\xrightarrow[]{k_3} X_3 + 2X_1\\
X_1 + X_2 &\xrightarrow[]{k_4} X_1 + X_2 + X_3\\
X_2 + X_3 &\xrightarrow[]{k_5} X_1 + X_2 + X_3\\
X_1 + X_3 &\xrightarrow[]{k_6} X_1 + X_2 + X_3
\end{split}
\end{eqnarray}

An E-graph that generates the dynamics of the \textit{relative} concentrations for this network is shown in Figure~\ref{fig:hypercycle}. This E-graph is endotactic, and therefore its dynamics can be embedded into a toric differential inclusion.

\begin{figure}[h]
\centering
\includegraphics[scale=0.45]{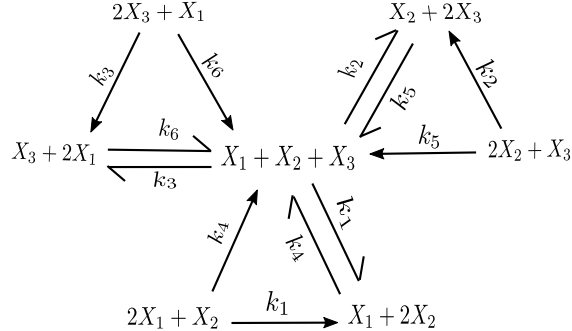}
\caption{\small Endotactic E-graph that generates the dynamics corresponding to the relative concentrations for network (\ref{eq:hypercycle}).}
\label{fig:hypercycle}
\end{figure}

\section{Acknowledgements}
G.C. is supported by NSF grants DMS-1412643 and DMS-1816238. A.D. acknowledges Van Vleck Visiting Assistant Professorship from the Department of Mathematics at University of Wisconsin Madison. We thank Polly Yu for suggesting several improvements to the presentation of these results.

\bibliographystyle{amsplain}
\bibliography{Bibliography}

\end{document}